\documentclass[12pt]{amsart}

\usepackage{graphicx}

\usepackage{url}

\DeclareSymbolFont{rmlargesymbols}{U}{euex}{m}{n}
\DeclareMathSymbol{\rmintop}{\mathop}{rmlargesymbols}{82}

\def\R{\mathbb{R}}

\def\cD{\mathcal{D}}

\def\al{\alpha}

\def\ga{\gamma}

\def\om{\omega}

\def\Om{\Omega}

\newcommand{\der}{{\rm d}}

\numberwithin{equation}{section}

\newtheorem{theorem}{Theorem}[section]

\theoremstyle{remark}

\usepackage{amssymb,stmaryrd}
\usepackage{amscd}

\newcommand{\qr}{
\begin{center}
\includegraphics[scale=0.5]{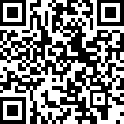}\\

\text{Scan the QR code to view more articles from the author}
\end{center}}

\author{Matthew Randall}
\address{Institute of Mathematical Sciences\\
ShanghaiTech University\\
393 Middle Huaxia Road\\
Shanghai, 201210\\
China}
\email{mjrandall@shanghaitech.edu.cn}


\title{Nurowski's conformal class of a maximally symmetric $(2,3,5)$-distribution and its Ricci-flat representatives}

\begin{document}

\begin{abstract}
We show that the solutions to the second-order differential equation associated to the generalised Chazy equation with parameters $k=2$ and $k=3$ naturally show up in the conformal rescaling that takes a representative metric in Nurowski's conformal class associated to a maximally symmetric $(2,3,5)$-distribution (described locally by a certain function $\varphi(x,q)=\frac{q^2}{H''(x)}$) to a Ricci-flat one. 
\end{abstract}

\maketitle

\pagestyle{myheadings}
\markboth{Randall}{Nurowski's conformal class of a maximally symmetric distribution and its Ricci-flat representatives}

The article concerns the occurrence of the $k=2$ and $k=3$ generalised Chazy equation in a geometric setting, closely connected to the occurrence of the solutions of the generalised Chazy equation with parameters $k=\frac{2}{3}$ and $k=\frac{3}{2}$ respectively. We first discuss the set-up in which the differential equations will appear. 
This concerns the theory of maximally non-integrable rank 2 distribution $\cD$ on a 5-manifold $M$. The maximally non-integrable condition of $\cD$ determines a filtration of the tangent bundle $TM$ given by
\[
\cD \subset [\cD,\cD] \subset [\cD,[\cD,\cD]]\cong TM.
\]
The distribution $[\cD, \cD]$ has rank 3 while the full tangent space $TM$ has rank 5, hence such a geometry is also known as a $(2,3,5)$-distribution. Let $M_{xyzpq}$ denote the 5-dimensional mixed order jet space $J^{2,0}(\R,\R^2) \cong J^2(\R,\R)\times \R$ with local coordinates given by $(x,y,z,p,q)=(x,y,z,y',y'')$ (see also \cite{tw13}, \cite{tw14}). Let $\cD_{\varphi(x,y,z,y',y'')}$ denote the maximally non-integrable rank 2 distribution on $M_{xyzpq}$ associated to the underdetermined differential equation $z'=\varphi(x,y,z,y',y'')$. This means that the distribution is annihilated by the following three 1-forms
\begin{align*}
\om_1=\der y-p \der x, \qquad \om_2=\der p-q \der x, \qquad  \om_3=\der z-\varphi(x,y,z,p,q) \der x.
\end{align*}
In \cite{conf}, it is shown how to associate canonically to such a (2,3,5)-distribution a conformal class of metrics of split signature $(2,3)$ (henceforth known as Nurowski's conformal structure or Nurowski's conformal metrics) such that the rank 2 distribution is isotropic with respect to any metric in the conformal class. The method of equivalence \cite{annur}, \cite{cartan1910}, \cite{conf}, \cite{Strazzullo}, \cite{new} gives the coframing for Nurowski's metric. The 1-forms in the coframe satisfy the structure equations
\begin{align}\label{cse}
\der \theta_1&=\theta_1\wedge (2\Om_1+\Om_4)+\theta_2\wedge \Om_2+\theta_3 \wedge \theta_4,\nonumber\\
\der \theta_2&=\theta_1\wedge\Om_3+\theta_2\wedge (\Om_1+2\Om_4)+\theta_3 \wedge \theta_5,\nonumber\\
\der \theta_3&=\theta_1\wedge\Om_5+\theta_2\wedge\Om_6+\theta_3\wedge (\Om_1+\Om_4)+\theta_4 \wedge \theta_5,\\
\der \theta_4&=\theta_1\wedge\Om_7+\frac{4}{3}\theta_3\wedge\Om_6+\theta_4\wedge \Om_1+\theta_5 \wedge \Om_2,\nonumber\\
\der \theta_5&=\theta_2\wedge \Om_7-\frac{4}{3}\theta_3\wedge \Om_5+\theta_4\wedge\Om_3+\theta_5\wedge \Om_4.\nonumber
\end{align}

A representative metric in Nurowski's conformal class \cite{conf} is given by
\begin{align}\label{metric}
g=2 \theta_1 \theta_5-2\theta_2 \theta_4+\frac{4}{3}\theta_3 \theta_3.
\end{align}
When $g$ has vanishing Weyl tensor, the distribution is called maximally symmetric and has split $G_2$ as its group of local symmetries.

The historic example is the case where $\varphi$ is given by $\varphi(x,y,z,p,q)=q^2$. When $\varphi(x,y,z,p,q)=q^2$, we obtain the Hilbert-Cartan distribution associated to the Hilbert-Cartan equation $z'=(y'')^2$.
When $\varphi(x,y,z,p,q)=q^m$, we obtain the distribution associated to the equation $z'=(y'')^m$. For such distributions, Nurowski's metric \cite{conf} given by (\ref{metric}) has vanishing Weyl tensor precisely when $m \in \{-1,\frac{1}{3},\frac{2}{3},2\}$. In these cases the maximally symmetric distributions are all locally diffeomorphic to the Hilbert-Cartan or flat model obtained when $m=2$.

In this article, we consider distributions here of the form $\varphi(x,y,z,p,q)=\frac{q^2}{H''(x)}$. The Weyl tensor vanishes in the case where $H(x)$ satisfies the 6th-order ordinary differential equation (ODE) known as Noth's equation \cite{annur}. For such maximally symmetric distributions we find the corresponding Ricci-flat representatives in Nurowki's conformal class. This involves solving a second-order differential equation (see Proposition 35 of \cite{tw13}) to find the conformal scale in which the Ricci tensor of the conformally rescaled metric vanishes, which turns out to be related to the solutions of Noth's equation. The 6th-order ODE can be reduced to the generalised Chazy equation with parameter $k=\frac{3}{2}$ and its Legendre dual is another 6th-order ODE that can be reduced to the generalised Chazy equation with parameter $k=\frac{2}{3}$. We find the second-order differential equation that determines the conformal scale for Ricci-flatness involves solutions of the generalised Chazy equation with parameter $k=3$ and in the dual case $k=2$. This is the content of Theorems \ref{A} and \ref{B}.
We also give few remarks concerning the case for other parameters of $k$.

The aim of finding Ricci-flat representatives is motivated by the consideration that in the Ricci-flat, conformally flat case, we might be able to integrate the structure equations and reexpress them in Monge normal form to obtain the Hilbert-Cartan distribution. This is possible for the distributions of the form $\varphi(x,y,z,p,q)=q^m$, with $m \in \{-1,\frac{1}{3},\frac{2}{3},2\}$, but would require further investigations in the general setting. 

The computations here are done using the indispensable \texttt{DifferentialGeometry} package in Maple 2018. 

\section{Deriving the equation for Ricci-flatness}
We shall consider the rank 2 distribution $\cD_{\varphi(x,q)}$ on $M_{xyzpq}$ associated to the underdetermined differential equation $z'=\varphi(x,y'')$ where $\varphi(x,y'')=\frac{(y'')^2}{H''(x)}$ and $H''(x)$ is a non-zero function of $x$. This is to say that the distribution $\cD_{\varphi(x,q)}$ is annihilated by the three 1-forms
\begin{align*}
\om_1=&\der y-p \der x,\nonumber\\
\om_2=&\der p-q \der x,\\
\om_3=&\der z-\varphi(x,q) \der x,\nonumber
\end{align*}
where $\varphi(x,q)=\frac{q^2}{H''(x)}$. These three 1-forms are completed to a coframing on $M_{xyzpq}$ by the additional 1-forms
\begin{align*}
\om_4=\der q-\frac{H^{(3)}}{H''}q\der x, \hspace{12pt} \om_5=-\frac{H''}{2}\der x.
\end{align*}
Taking appropriate linear combinations, we let 
\begin{align*}
\theta_1=&\om_3-\frac{2}{H''}q\om_2,\qquad \theta_2=\om_1, \qquad \theta_3=\left(\frac{2}{H''}\right)^{\frac{1}{3}}\om_2,
\end{align*}
with 
\[
\theta_4=\left(\frac{2}{H''}\right)^{\frac{2}{3}}\om_4+a_{41}\theta_1+a_{42}\theta_2+a_{43}\theta_3
\]
and 
\[
\theta_5=\left(\frac{2}{H''}\right)^{\frac{2}{3}}\om_5+a_{51}\theta_1+a_{52}\theta_2+a_{53}\theta_3.
\]
Imposing Cartan's structure equations (\ref{cse}) on $(\theta_1, \theta_2, \theta_3, \theta_4, \theta_5)$ then gives the constraints $a_{51}=a_{53}=0$ and $a_{41}=a_{52}$, which we can set both to be zero, and we also find
$a_{42}=\frac{1}{30}\frac{2^{\frac{2}{3}}(3H''H^{(4)}-5(H^{(3)})^2)}{(H'')^{\frac{8}{3}}}$
and
$a_{43}=-\frac{2^{\frac{1}{3}}H^{(3)}}{3(H'')^{\frac{4}{3}}}$. The metric $g=2\theta_1\theta_5-2\theta_2\theta_4+\frac{4}{3}\theta_3\theta_3$ is conformally flat, i.e. the metric $g$ has vanishing Weyl tensor if and only if $H(x)$ is a solution to the 6th-order nonlinear differential equation 
\begin{align}\label{noth}
10(H'')^3H^{(6)}-70(H'')^2H^{(3)}H^{(5)}-49(H'')^2(H^{(4)})^2+280H''(H^{(3)})^2H^{(4)}-175(H^{(3)})^4=0.
\end{align}
This equation is called Noth's equation \cite{annur}. 
In this case the distribution of the form $\cD_{\varphi(x,q)}$ is maximally symmetric and in the paper we will concern ourselves with the problem of finding Ricci-flat representatives in the conformal class of metrics associated to this distribution. 

The explicit form of the metric given by the distribution $\cD_{\varphi(x,q)}$ is as follows. If we replace $H''(x)=e^{\int \frac{2}{3}P(x)\der x}$, we find that equation (\ref{noth}) reduces to the $k=\frac{3}{2}$ generalised Chazy equation
\[
P'''-2PP''+3P'^2-\frac{4}{36-(\frac{3}{2})^2}(6P'-P^2)^2=0,
\]
and we find that the conformally rescaled metric
$\tilde g=2^{-\frac{2}{3}}(H'')^{\frac{2}{3}}g$ has the form
\begin{align*}
\tilde g=-\frac{2}{15}(P'-\frac{4}{9}P^2)\om_1\om_1+\frac{4}{9}P\om_1\om_2+\frac{4}{3}\om_2\om_2+2\om_3\om_5-2\om_1\om_4-4qe^{\int -\frac{2}{3}P \der x}\om_2\om_5.
\end{align*}
We can reexpress this metric as
\begin{align*}
\tilde g&=-\frac{2}{15}(P'-\frac{1}{6}P^2)\om_1\om_1+\frac{4}{3}\left(\frac{P}{6}\om_1+\om_2\right)\left(\frac{P}{6}\om_1+\om_2\right)+2\om_3\om_5-2\om_1\om_4-4qe^{\int -\frac{2}{3}P \der x}\om_2\om_5.
\end{align*}
By defining the new coframes
\begin{align*}
\tilde \om_3&=e^{\int \frac{2P}{3}\der x}\om_3,\\
\tilde \om_5&=e^{-\int \frac{2P}{3}\der x}\om_5,
\end{align*} and making the further substitution $Q=P^2-6P'$,
we get the following cosmetic improvement for $\tilde g$: 
\begin{align*}
\tilde g
&=\frac{1}{45}Q\om_1\om_1+\frac{4}{3}\left(\frac{P}{6}\om_1+\om_2\right)\left(\frac{P}{6}\om_1+\om_2\right)+2\tilde \om_3\tilde \om_5-2\om_1\om_4-4q \om_2\tilde \om_5.
\end{align*}
From this we can rescale the metric $\tilde g$ further by a conformal factor $\Om$ to obtain a Ricci-flat representative. When ${\rm Ric}(\Om^2\tilde g)=0$, we say that $\Om^2\tilde g$ is a Ricci-flat representative of Nurowski's conformal class. 
We find that $\Om^2 \tilde g$ is Ricci-flat when $\Om$ satisfies the second-order differential equation
\begin{equation*}
\Om''\Om-2(\Om')^2-\frac{2}{3}P\Om\Om'-\frac{1}{18}P^2\Om^2-\frac{1}{30}Q\Om^2=0.
\end{equation*}
We make the substitution $\Om=\frac{1}{\rho} e^{-\frac{1}{3}\int P\der x}$ to obtain
\begin{equation}\label{sode1}
\rho''-\frac{1}{45}Q\rho=0,
\end{equation}
where $\rho(x)$ is to be determined.

The function $H(x)$ is related to another function $F(\tilde x)$ by a Legendre transformation \cite{annur}, \cite{r16}. We say that $F(\tilde x)$ is the Legendre dual of $H(x)$ determined by the relation $H(x)+F(\tilde x)=x\tilde x$. This implies $\tilde x=H'(x)$ with $\der \tilde x=H''\der x$ and $H''=\frac{1}{F_{\tilde x\tilde x}}$. We can make use of this transformation to write $\der x=F_{\tilde x\tilde x}\der \tilde x$. The Legendre dual of the distribution $\cD_{\varphi(x,q)}$ is therefore given by the annihilator of the three 1-forms
\begin{align*}
\om_1=&\der y-p F_{\tilde x \tilde x}\der \tilde x,\nonumber\\
\om_2=&\der p-q F_{\tilde x \tilde x}\der \tilde x,\\
\om_3=&\der z-q^2 (F_{\tilde x \tilde x})^2\der \tilde x\nonumber
\end{align*}
on the mixed jet space with local coordinates $(\tilde x, y, z, p, q)$. Relabelling $\tilde x$ with $x$, we have
\begin{align*}
\om_1=&\der y-p F''\der x,\nonumber\\
\om_2=&\der p-q F''\der x,\\
\om_3=&\der z-q^2 (F'')^2\der x.\nonumber
\end{align*}
Here $F$ now becomes a function of $x$. These three 1-forms are completed to a coframing on $M$ with local coordinates $(x,y,z,p,q)$ by the additional 1-forms
\begin{align*}
\om_4=\der q+\frac{F'''}{F''}q\der x, \hspace{12pt} \om_5=-\frac{1}{2}\der x.
\end{align*}
(These are the Legendre transformed 1-forms $\om_4$ and $\om_5$). Similar as before, we consider the linear combinations
\begin{align*}
\theta_1=\om_3-2F'' q\om_2,\qquad \theta_2=\om_1, \qquad \theta_3=(2 F'')^{\frac{1}{3}}\om_2,
\end{align*}
with 
\[
\theta_4=(2 F'')^{\frac{2}{3}}\om_4+b_{41}\theta_1+b_{42}\theta_2+b_{43}\theta_3
\]
and 
\[
\theta_5=(2 F'')^{\frac{2}{3}}\om_5+b_{51}\theta_1+b_{52}\theta_2+b_{53}\theta_3.
\]
Imposing Cartan's structure equations (\ref{cse}) on $(\theta_1, \theta_2, \theta_3, \theta_4, \theta_5)$ again gives $b_{51}=b_{53}=0$ and $b_{41}=b_{52}$, which we set to be zero. We also obtain $b_{42}=-\frac{1}{30}\frac{2^{\frac{2}{3}}(3F''F^{(4)}-4(F^{(3)})^2)}{(F'')^{\frac{10}{3}}}$ and
$b_{43}=\frac{2^{\frac{1}{3}}F^{(3)}}{3(F'')^{\frac{5}{3}}}$. A representative metric of Nurowski's conformal class is again given by (\ref{metric}).
The condition that the metric $g$ is conformally flat, i.e.\ the metric $g$ has vanishing Weyl tensor, occurs when $F(x)$ is a solution to the nonlinear differential equation 
\begin{align}\label{6thode}
10(F'')^3F^{(6)}-80(F'')^2F^{(3)}F^{(5)}-51(F'')^2(F^{(4)})^2+336F''(F^{(3)})^2F^{(4)}-224(F^{(3)})^4=0.
\end{align}
If we replace $F''(x)=e^{\int \frac{1}{2}P(x)\der x}$, we find that the conformally rescaled metric
$\tilde g=2^{\frac{1}{3}}(F'')^{-\frac{2}{3}}g$ has the form
\begin{align}\label{roughg}
\tilde g=\frac{1}{30}(6P'-P^2)e^{\int -P \der x}\om_1\om_1-\frac{2}{3}Pe^{-\int \frac{1}{2}P \der x}\om_1\om_2+\frac{8}{3}\om_2\om_2+4\om_3\om_5-4\om_1\om_4-8qe^{\int \frac{1}{2}P \der x}\om_2\om_5.
\end{align}
Here equation (\ref{6thode}) is reduced to the generalised Chazy equation 
\[
P'''-2PP''+3P'^2-\frac{4}{36-(\frac{2}{3})^2}(6P'-P^2)^2=0
\]
for $P(x)$ with parameter $k=\frac{2}{3}$.
From the form of the metric $\tilde g$ we can locally rescale the metric again by a conformal factor to obtain Ricci-flat representatives. 

We find that the Ricci tensor of $\Om^2 \tilde g$ is zero when $\Om$ satisfies
\begin{equation*}
40\Om''\Om-80(\Om')^2-6\Om^2P'+\Om^2P^2=0.
\end{equation*}
If we make the substitution $\Om=\frac{1}{\nu}$, then we obtain the differential equation
\begin{equation}\label{sode2}
\nu''-\frac{1}{40}Q\nu=0
\end{equation}
where $Q=P^2-6P'$ and $\nu$ is to be determined. From the form the metric $\tilde g$ in (\ref{roughg}), we can also define new coframes by
\begin{align*}
\tilde \om_1&=e^{-\int \frac{P}{2}\der x}\om_1=\frac{\der y}{F''}-p\der x,\nonumber\\
\tilde \om_2&=\om_2=\der p-q F'' \der x,\nonumber\\
\tilde \om_3&=e^{-\int \frac{P}{2}\der x}\om_3=\frac{\der z}{F''}-q^2 F''\der x,\\
\tilde \om_4&=e^{\int \frac{P}{2}\der x}\om_4=F'' \der q+q F''' \der x,\nonumber\\
\tilde \om_5&=e^{\int \frac{P}{2}\der x}\om_5=-\frac{F''}{2}\der x.\nonumber
\end{align*}
We have used that $e^{-\int \frac{P}{2}\der x}=\frac{1}{F''}$.
Also replacing $6P'-P^2=-Q$, this gives the cosmetic improvement for $\tilde g$: 
\begin{align*}
\tilde g=-\frac{Q}{30}\tilde\om_1\tilde\om_1-\frac{2P}{3}\tilde\om_1\tilde\om_2+\frac{8}{3}\tilde\om_2\tilde\om_2+4\tilde\om_3\tilde\om_5-4\tilde\om_1\tilde\om_4-8q\tilde\om_2\tilde\om_5.
\end{align*}
We now investigate the solutions to (\ref{sode1}) and (\ref{sode2}). They are given by Theorems \ref{A} and \ref{B}. We first review some results about the solutions to the generalised Chazy equation. 

\section{Generalised Chazy equation}

The generalised Chazy equation with parameter $k$ is given by
\begin{align*}
y'''-2yy''+3y'^2-\frac{4}{36-k^2}(6y'-y^2)^2=0
 \end{align*} 
and Chazy's equation 
\begin{align*}
y'''-2yy''+3y'^2=0
 \end{align*} 
is obtained in the limit as $k$ tends to infinity. The generalised Chazy equation was introduced in \cite{chazy1}, \cite{chazy2} and studied more recently in \cite{co96}, \cite{acht}, \cite{ach} and \cite{bc17}. The generalised Chazy equation with parameters $k=\frac{2}{3}$, $\frac{3}{2}$, $2$ and $3$ was also further investigated in \cite{r16b}. The solution to the generalised Chazy equation is given by the following (see also \cite{bc17} and \cite{r16b}). Let 
\begin{align*}\label{weq}
w_1&=-\frac{1}{2}\frac{\der }{\der x}\log\frac{s'}{s(s-1)},\nonumber \\
w_2&=-\frac{1}{2}\frac{\der }{\der x}\log\frac{s'}{s-1},\\
w_3&=-\frac{1}{2}\frac{\der }{\der x}\log\frac{s'}{s},\nonumber
\end{align*}
where $s=s(\al,\beta,\ga,x)$ is a solution to the Schwarzian differential equation 
\begin{equation}\label{sde}
\{s,x\}+\frac{1}{2}(s')^2V=0
\end{equation}
and
\[
\{s,x\}=\frac{\der}{\der x}\left(\frac{s''}{s'}\right)-\frac{1}{2}\left(\frac{s''}{s'}\right)^2
\]
is the Schwarzian derivative with the potential $V$ given by
\begin{equation}\label{pot}
V=\frac{1-\beta^2}{s^2}+\frac{1-\ga^2}{(s-1)^2}+\frac{\beta^2+\ga^2-\al^2-1}{s(s-1)}.
\end{equation}
The combination $y=-2w_1-2w_2-2w_3$ solves the generalised Chazy equation when 
\[
(\al,\beta,\ga)=\left(\frac{1}{3},\frac{1}{3},\frac{2}{k}\right) \text{~or~} \left(\frac{2}{k},\frac{2}{k},\frac{2}{k}\right).
\]

The combination $y=-w_1-2w_2-3w_3$ solves the generalised Chazy equation when 
\[
(\al,\beta,\ga)=\left(\frac{1}{k},\frac{1}{3},\frac{1}{2}\right) \text{~or~} \left(\frac{1}{k},\frac{2}{k},\frac{1}{2}\right) \text{~or~}  \left(\frac{1}{k},\frac{1}{3},\frac{3}{k}\right) ,
\]
with permutations of $w_1$, $w_2$ and $w_3$ in $y$ corresponding to permutations of the values $\al$, $\beta$ and $\ga$ in $(\al,\beta,\ga)$.
The combination $y=-w_1-w_2-4w_3$ solves the generalised Chazy equation whenever 
\[
(\al,\beta,\ga)=\left(\frac{1}{k},\frac{1}{k},\frac{4}{k}\right) \text{~or~} \left(\frac{1}{k},\frac{1}{k},\frac{2}{3}\right),
\]
again permuting $w_1$, $w_2$ and $w_3$ in $y$ corresponds to permuting the values $\al$, $\beta$, $\ga$ in $(\al,\beta,\ga)$.
Following \cite{acht}, the functions $w_1$, $w_2$ and $w_3$ satisfy the following system of differential equations:
\begin{align}\label{wde}
w_1'=w_2w_3-w_1(w_2+w_3)+\tau^2,\nonumber\\
w_2'=w_3w_1-w_2(w_3+w_1)+\tau^2,\\
w_3'=w_1w_2-w_3(w_1+w_2)+\tau^2,\nonumber
\end{align}
where
\[
\tau^2=\al^2(w_1-w_2)(w_3-w_1)+\beta^2(w_2-w_3)(w_1-w_2)+\ga^2(w_3-w_1)(w_2-w_3).
\]
The second-order differential equation associated to the generalised Chazy equation with parameter $k$ is given by
\begin{align}\label{uode}
u_{ss}+\frac{1}{4}Vu=0
\end{align}
with the same potential $V$ as given in (\ref{pot}).
This corresponds to the general solution of the Schwarzian differential equation (\ref{sde}) after exchanging dependent and independent variables \cite{co96}. In this case $x=\frac{u_2}{u_1}$ where $u_1$ and $u_2$ are linearly independent solutions to (\ref{uode}). Using the further substitution $u(s)=(s-1)^{\frac{1-\gamma}{2}}s^{\frac{1-\beta}{2}}z(s)$, the equation (\ref{uode}) can be brought to the hypergeometric differential equation
\begin{align*}
s(1-s)z_{ss}+(c-(a+b+1)s)z_s-a bz=0
\end{align*}
with 
\[
a=\frac{1}{2}(1-\al-\beta-\ga), \qquad b=\frac{1}{2}(1+\al-\beta-\ga), \qquad c=1-\beta.
\]
From the differential equations (\ref{wde}), we can recover $s$ by $s=\frac{w_1-w_3}{w_2-w_3}$. From this we deduce $s'=2(w_1-w_2)s$ and we also obtain the relation $\der s=2(w_1-w_2) s \der x$.

\section{Main results: Solving the equations for Ricci-flatness}
In this section we give the general solution to the differential equation (\ref{sode1})
where $Q=P^2-6P'$ and $P$ is a solution of the $k=\frac{3}{2}$ generalised Chazy equation in Theorem \ref{A} and
the general solution to the differential equation (\ref{sode2})
where again $Q=P^2-6P'$ and $P$ is a solution of the $k=\frac{2}{3}$ generalised Chazy equation in Theorem \ref{B}.
We first prove the following theorem
\begin{theorem}\label{A}
The solution to the differential equation
 \[
\rho''-\frac{1}{45}Q\rho=0,
\]
where $Q=P^2-6P'$ and $P$ is a solution to the $k=\frac{3}{2}$ generalised Chazy equation, is given by $\rho=\frac{u}{v}$ where
$v$ is the solution to the second-order differential equation associated to the $k=\frac{3}{2}$ generalised Chazy equation and $u$ is a solution to the second-order differential equation associated to the $k=3$ generalised Chazy equation. 
\end{theorem}
\begin{proof}
To prove the claim, we consider the second-order differential equation of the form $v_{ss}+\frac{1}{4}Vv=0$ associated to the generalised Chazy equation with parameter $k=\frac{3}{2}$, where $V$ is the function given by
\begin{align*}
V=\frac{1-\beta^2}{s^2}+\frac{1-\ga^2}{(s-1)^2}+\frac{\beta^2+\ga^2-\al^2-1}{s(s-1)}.
\end{align*}
We find that $v=v(s(x))$ as a function of $x$ satisfies
\begin{align*}
&v_{xx}-2(w_1-w_2-w_3)v_x-((\al^2-1)w_1^2+(\beta^2-1)w_2^2+(\ga^2-1)w_3^2)v\\
&+((\al^2+\beta^2-\ga^2-1)w_1w_2+(\al^2-\beta^2+\ga^2-1)w_1w_3-(\al^2-\beta^2-\ga^2+1)w_2w_3)v=0.
\end{align*}
We have used that
\begin{align*}
\frac{\der}{\der s}=\frac{(w_2-w_3)}{2(w_1-w_2)(w_1-w_3)}\frac{\der}{\der x}
\end{align*}
and the differential equations (\ref{wde}).
Furthermore, the Wronskian $W=v_1(v_2)_s-v_2(v_1)_s$ of the solutions to this differential equation satisfies $W_s=0$, so $W=c_0$ and we have
\begin{align*}
v_1^2=2c_0(w_1-w_2)s
\end{align*}
from the consideration that $s'=2(w_1-w_2)s=\frac{v_1^2}{W}$,
and furthermore we obtain from the differential equation that the Wronskian $W=\frac{v(s(x))^2}{2(w_1-w_2)s(x)}$ satisfies, that
\begin{align}\label{vx1}
v_{x}-v(w_1-w_2-w_3)=0.
\end{align}
This equation implies the differential equation for $v$ above, by using the fact that $w_i$ satisfy the differential equations (\ref{wde}).

Upon making the substitution $\rho=\frac{u(x)}{v(x)}$ into equation (\ref{sode1}), and using equation (\ref{vx1}), we obtain a differential equation for $u(x)$ remaining. This differential equation for $u(x)$ turns out to be of the form 
\begin{align*}
&u_{xx}-2(w_1-w_2-w_3)u_x-((\tilde \al^2-1)w_1^2+(\tilde \beta^2-1)w_2^2+(\tilde \ga^2-1)w_3^2)u\\
&+((\tilde \al^2+\tilde \beta^2-\tilde \ga^2-1)w_1w_2+(\tilde \al^2-\tilde \beta^2+\tilde \ga^2-1)w_1w_3-(\tilde \al^2-\tilde \beta^2-\tilde \ga^2+1)w_2w_3)u=0,
\end{align*}
which is the same differential equation for $v$ with different constants $\tilde \al$, $\tilde \beta$, $\tilde \ga$. This is the differential equation associated to the generalised Chazy equation with parameter $k=3$. We see this automatically when we compute the values with $Q=P^2-6P'$ where $P$ is the solution of the generalised Chazy equation with parameter $k=\frac{3}{2}$. Specialising to the case where $k=\frac{3}{2}$, we obtain the following:

For the solutions given by $P=-2w_1-2w_2-2w_3$, when $(\al,\beta,\ga)=(\frac{4}{3},\frac{4}{3},\frac{4}{3})$, we find $(\tilde\al,\tilde\beta,\tilde\ga)=(\frac{2}{3},\frac{2}{3},\frac{2}{3})$.
When $(\al,\beta,\ga)=(\frac{4}{3},\frac{1}{3},\frac{1}{3})$, we find $(\tilde\al,\tilde\beta,\tilde\ga)=(\frac{2}{3},\frac{1}{3},\frac{1}{3})$.

For the solutions given by $P=-w_1-2w_2-3w_3$, when $(\al,\beta,\ga)=(\frac{2}{3},\frac{1}{3},\frac{1}{2})$, we find $(\tilde\al,\tilde\beta,\tilde\ga)=(\frac{1}{3},\frac{1}{3},\frac{1}{2})$. When $(\al,\beta,\ga)=(\frac{2}{3},\frac{4}{3},\frac{1}{2})$, we find $(\tilde\al,\tilde\beta,\tilde\ga)=(\frac{1}{3},\frac{2}{3},\frac{1}{2})$. When $(\al,\beta,\ga)=(\frac{2}{3},\frac{1}{3},2)$, we find $(\tilde\al,\tilde\beta,\tilde\ga)=(\frac{1}{3},\frac{1}{3},1)$.

Finally for the solutions given by $P=-4w_1-w_2-w_3$, when $(\al,\beta,\ga)=(\frac{2}{3},\frac{2}{3},\frac{2}{3})$, we find $(\tilde\al,\tilde\beta,\tilde\ga)=(\frac{2}{3},\frac{1}{3},\frac{1}{3})$. When $(\al,\beta,\ga)=(\frac{8}{3},\frac{2}{3},\frac{2}{3})$, we find $(\tilde\al,\tilde\beta,\tilde\ga)=(\frac{4}{3},\frac{1}{3},\frac{1}{3})$. 

The values $(\tilde\al,\tilde\beta,\tilde\ga)$ are precisely the ones that show up in the solutions of the $k=3$ generalised Chazy equation. See \cite{r16b} for the list of $(\tilde\al,\tilde\beta,\tilde\ga)$ when $k=3$.

\end{proof}

The determination of solutions to equation (\ref{sode2}) is similar to that of Theorem \ref{A}. We prove the following

\begin{theorem}\label{B}
The solution to the differential equation
\begin{equation}\label{rf0a}
\nu''-\frac{1}{40}Q\nu=0,
\end{equation}
where $Q=P^2-6P'$ and $P$ is a solution of the $k=\frac{2}{3}$ generalised Chazy equation, is given by $\nu=\frac{u}{v}$, where $v$ is a solution to the second-order differential equation associated to the $k=\frac{2}{3}$ generalised Chazy equation and $u$ is a solution to the second-order differential equation associated to the $k=2$ generalised Chazy equation.
\end{theorem}
\begin{proof}
The proof of the claim is similar to the proof of the previous theorem. From the differential equation of the form $v_{ss}+\frac{1}{4}Vv=0$ associated to the $k=\frac{2}{3}$ generalised Chazy equation, where $V$ is the function given by
\begin{align*}
V=\frac{1-\beta^2}{s^2}+\frac{1-\ga^2}{(s-1)^2}+\frac{\beta^2+\ga^2-\al^2-1}{s(s-1)},
\end{align*}
we find that $v=v(s(x))$ as a function of $x$ satisfies
\begin{align*}
&v_{xx}-2(w_1-w_2-w_3)v_x-((\al^2-1)w_1^2+(\beta^2-1)w_2^2+(\ga^2-1)w_3^2)v\\
&+((\al^2+\beta^2-\ga^2-1)w_1w_2+(\al^2-\beta^2+\ga^2-1)w_1w_3-(\al^2-\beta^2-\ga^2+1)w_2w_3)v=0.
\end{align*}
Like in the proof of Theorem \ref{A}, it can also be deduced that (\ref{vx1}) holds for $v$, i.e.
\begin{align}\label{vx2}
v_{x}-v(w_1-w_2-w_3)=0, 
\end{align}
which again implies the differential equation for $v$ above, by using the fact that $w_i$ satisfy the differential equations (\ref{wde}).

Upon making the substitution $\nu=\frac{u(x)}{v(x)}$ into equation (\ref{rf0a}), and using equation (\ref{vx2}), we obtain a differential equation for $u(x)$ remaining. The differential equation for $u(x)$ is again
\begin{align}\label{ustar}
&u_{xx}-2(w_1-w_2-w_3)u_x-((\tilde \al^2-1)w_1^2+(\tilde \beta^2-1)w_2^2+(\tilde \ga^2-1)w_3^2)u\\
&+((\tilde \al^2+\tilde \beta^2-\tilde \ga^2-1)w_1w_2+(\tilde \al^2-\tilde \beta^2+\tilde \ga^2-1)w_1w_3-(\tilde \al^2-\tilde \beta^2-\tilde \ga^2+1)w_2w_3)u=0,\nonumber
\end{align}
which is the same differential equation for $v$ but with different constants $\tilde \al$, $\tilde \beta$, $\tilde \ga$. The equation (\ref{ustar}) corresponds to the second-order differential equation associated to the $k=2$ generalised Chazy equation. To see this, we shall compute these constants when $Q=P^2-6P'$ and $P$ is the solution of the generalised Chazy equation with parameter $k=\frac{2}{3}$. Specialising to the case where $k=\frac{2}{3}$, we obtain the following:

For the solutions given by $P=-2w_1-2w_2-2w_3$, when $(\al,\beta,\ga)=(3,3,3)$, we find $(\tilde\al,\tilde\beta,\tilde\ga)=(1,1,1)$.
When $(\al,\beta,\ga)=(3,\frac{1}{3},\frac{1}{3})$, we find $(\tilde\al,\tilde\beta,\tilde\ga)=(1,\frac{1}{3},\frac{1}{3})$.

For the solutions given by $P=-w_1-2w_2-3w_3$, when $(\al,\beta,\ga)=(\frac{3}{2},\frac{1}{3},\frac{1}{2})$, we find $(\tilde\al,\tilde\beta,\tilde\ga)=(\frac{1}{2},\frac{1}{3},\frac{1}{2})$. When $(\al,\beta,\ga)=(\frac{3}{2},3,\frac{1}{2})$, we find $(\tilde\al,\tilde\beta,\tilde\ga)=(\frac{1}{2},1,\frac{1}{2})$. When $(\al,\beta,\ga)=(\frac{3}{2},\frac{1}{3},\frac{9}{2})$, we find $(\tilde\al,\tilde\beta,\tilde\ga)=(\frac{1}{2},\frac{1}{3},\frac{3}{2})$.

Finally for the solutions given by $P=-4w_1-w_2-w_3$, when $(\al,\beta,\ga)=(\frac{2}{3},\frac{3}{2},\frac{3}{2})$, we find $(\tilde\al,\tilde\beta,\tilde\ga)=(\frac{2}{3},\frac{1}{2},\frac{1}{2})$. When $(\al,\beta,\ga)=(6,\frac{3}{2},\frac{3}{2})$, we find $(\tilde\al,\tilde\beta,\tilde\ga)=(2,\frac{1}{2},\frac{1}{2})$. 

The values $(\tilde\al,\tilde\beta,\tilde\ga)$ are precisely the ones that show up in the solutions of the $k=2$ generalised Chazy equation. See also \cite{r16b} for the list of $(\tilde\al,\tilde\beta,\tilde\ga)$ when $k=2$.

\end{proof}

\section{Solution to the equation for Ricci-flatness for general Chazy parameter}
More generally, when $P$ is a solution to the generalised Chazy equation with parameter $k$, the metric $g$ is no longer conformally flat but we can still find the conformal scale for which the Ricci tensor vanishes.

In the case of (\ref{sode1}) with solutions given by $\nu=\frac{u}{v}$ where $v$ is the second-order differential equation associated to the generalised Chazy equation with parameter $k$, we find that $u$ is a solution to the second-order differential equation associated to the generalised Chazy equation with parameter $\tilde k$ with
 \begin{equation}\label{g1}
\frac{45}{\tilde k^2}-\frac{9}{k^2}=1. 
\end{equation}
The values $(\al, \beta, \ga)$ appearing in $V$ in the differential equation $v_{ss}+\frac{1}{4}Vv=0$ are related to the values $(\tilde \al, \tilde \beta, \tilde \ga)$ appearing in $V$ in the differential equation $u_{ss}+\frac{1}{4}Vu=0$ by the following. For the solutions given by $P=-2w_1-2w_2-2w_3$, when $(\al,\beta,\ga)=(\frac{2}{k},\frac{2}{k},\frac{2}{k})$, we find $(\tilde\al,\tilde\beta,\tilde\ga)$ with 
\begin{align*}
\frac{45}{4}\tilde\al^2-\left(\frac{3}{k}\right)^2=1,\hspace{12pt}
\frac{45}{4}\tilde\beta^2-\left(\frac{3}{k}\right)^2=1,\hspace{12pt}
\frac{45}{4}\tilde\ga^2-\left(\frac{3}{k}\right)^2=1.
\end{align*}

When $(\al,\beta,\ga)=(\frac{2}{k},\frac{1}{3},\frac{1}{3})$, we find $(\tilde\al,\tilde\beta,\tilde\ga)$ with
\begin{align*}
\frac{45}{4}\tilde\al^2-\left(\frac{3}{k}\right)^2=1
\end{align*}
and $\tilde \beta=\frac{1}{3}$, $\tilde \ga=\frac{1}{3}$. Here and subsequently, we shall consider the positive square root that gives positive $\tilde \al$, $\tilde \beta$ and $\tilde \ga$. 

For the solutions given by $P=-w_1-2w_2-3w_3$, when $(\al,\beta,\ga)=(\frac{1}{k},\frac{1}{3},\frac{1}{2})$, we find $(\tilde\al,\tilde\beta,\tilde\ga)$ with
\begin{align*}
45\tilde\al^2-\left(\frac{3}{k}\right)^2=1,\hspace{12pt}
\tilde\beta=\frac{1}{3},\hspace{12pt}
\tilde\ga=\frac{1}{2}.
\end{align*}

When $(\al,\beta,\ga)=(\frac{1}{k},\frac{2}{k},\frac{1}{2})$, we find $(\tilde\al,\tilde\beta,\tilde\ga)$ with
\begin{align*}
45\tilde\al^2-\left(\frac{3}{k}\right)^2=1,\hspace{12pt}
\frac{45}{4}\tilde\beta^2-\left(\frac{3}{k}\right)^2=1,\hspace{12pt}
\tilde\ga=\frac{1}{2}.
\end{align*}

When $(\al,\beta,\ga)=(\frac{1}{k},\frac{1}{3},\frac{3}{k})$, we find $(\tilde\al,\tilde\beta,\tilde\ga)$ with
\begin{align*}
45\tilde\al^2-\left(\frac{3}{k}\right)^2=1,\hspace{12pt}
\tilde \beta=\frac{1}{3},\hspace{12pt}
5\tilde\ga-\left(\frac{3}{k}\right)^2=1.
\end{align*}

Finally for the solutions given by $P=-4w_1-w_2-w_3$, when $(\al,\beta,\ga)=(\frac{4}{k},\frac{1}{k},\frac{1}{k})$, we find $(\tilde\al,\tilde\beta,\tilde\ga)$ with
\begin{align*}
\frac{45}{16}\tilde\al^2-\left(\frac{3}{k}\right)^2=1,\hspace{12pt}
45\tilde\beta^2-\left(\frac{3}{k}\right)^2=1,\hspace{12pt}
45\tilde\ga^2-\left(\frac{3}{k}\right)^2=1.
\end{align*}

When $(\al,\beta,\ga)=(\frac{2}{3},\frac{1}{k},\frac{1}{k})$, we find $(\tilde\al,\tilde\beta,\tilde\ga)$ with 
\begin{align*}
\tilde\al=\frac{2}{3},\hspace{12pt}
45\tilde\beta^2-\left(\frac{3}{k}\right)^2=1,\hspace{12pt}
45\tilde\ga^2-\left(\frac{3}{k}\right)^2=1.
\end{align*}
In all cases the appropriate substitution of $\tilde \al$, $\tilde \beta$ and $\tilde \ga$ in terms of the Chazy parameter $\tilde k$ gives equation (\ref{g1}),
so it can be seen that the equation for $u$ is the second-order differential equation associated to the generalised Chazy equation with parameter $\tilde k$, related to $k$ by (\ref{g1}). The further substitution $k=\frac{3}{m}$ and $\tilde k=\frac{3}{\tilde m}$ into (\ref{g1}) gives
\[
5\tilde m^2-m^2=1,
\]
which has integer solutions when considered as a negative Pell equation. 
For integer solutions $m$ and $\tilde m$ we obtain
\begin{align*}
m&=\pm(\frac{1}{2}(2+\sqrt{5})^{2n+1}+\frac{1}{2}(2-\sqrt{5})^{2n+1}),\\
\tilde m&=\pm(\frac{\sqrt{5}}{10}(2+\sqrt{5})^{2n+1}-\frac{\sqrt{5}}{10}(2-\sqrt{5})^{2n+1}).
\end{align*}
They take on values $(m,\tilde m)=(2,1)$, $(38,17)$, $(682,305)$, $(12238,5473)$ and so on for $n\in \mathbb{N} \cup \{0\}$. They also give the corresponding pairs of Chazy parameters $(k, \tilde k)=(\frac{3}{2},3)$, $(\frac{3}{38},\frac{3}{17})$ and so on, with the fundamental solution ($n=0$) agreeing with the result of Theorem \ref{A} in the conformally flat case.

 In the case of (\ref{sode2}) with solutions given by $\nu=\frac{u}{v}$ where $v$ is the second-order differential equation associated to the generalised Chazy equation with parameter $k$, we find that $u$ is a solution to the second-order differential equation associated to the generalised Chazy equation with parameter $\tilde k$ with
 \begin{equation}\label{g2}
\frac{40}{\tilde k^2}-\frac{4}{k^2}=1. 
\end{equation}
In this case we obtain the relationship between the values $(\al, \beta, \ga)$ and $(\tilde \al, \tilde \beta, \tilde \ga)$ as follows. For $P=-2w_1-2w_2-2w_3$, when $(\al,\beta,\ga)=(\frac{2}{k},\frac{2}{k},\frac{2}{k})$, we find $(\tilde\al,\tilde\beta,\tilde\ga)$ with 
\begin{align*}
10\tilde\al^2-\left(\frac{2}{k}\right)^2=1,\hspace{12pt}
10\tilde\beta^2-\left(\frac{2}{k}\right)^2=1,\hspace{12pt}
10\tilde\ga^2-\left(\frac{2}{k}\right)^2=1.
\end{align*}
Considering integer solutions $\al$ and $\tilde \al$ to the negative Pell equation $10\tilde \al^2-\al^2=1$ (and also $\beta$, $\tilde \beta $ and $\ga$, $\tilde \ga$ respectively), we find
\begin{align*}
\al&=\pm(\frac{1}{2}(3+\sqrt{10})^{2n+1}+\frac{1}{2}(3-\sqrt{10})^{2n+1}),\\
\tilde \al&=\pm(\frac{\sqrt{10}}{20}(3+\sqrt{10})^{2n+1}-\frac{\sqrt{10}}{20}(3-\sqrt{10})^{2n+1}),
\end{align*}
where $n \in \mathbb{Z}$. Positive integer solutions are given by $(\al, \tilde \al)=(3,1)$, $(117,37)$, $(4443,1405)$, $(168717, 53353)$ and so on for $n \in \mathbb{N} \cup \{0\}$. They give the relationship between the pairs of Chazy parameters $k=\frac{2}{\al}$ and $\tilde k=\frac{2}{\tilde \al}$, with $(k, \tilde k)=(\frac{2}{3},2)$, $(\frac{2}{117},\frac{2}{37})$ and so on for $n \in \mathbb{N}\cup \{0\}$. For these parameters the associated hypergeometric functions are algebraic. Again the fundamental solution ($n=0$) agrees with the result of Theorem \ref{B} in the conformally flat case.

The determination of the other values of $(\al, \beta, \ga)$ and $(\tilde \al, \tilde \beta, \tilde \ga)$ are as follows. 
For the same $P$, when $(\al,\beta,\ga)=(\frac{2}{k},\frac{1}{3},\frac{1}{3})$, we find $(\tilde\al,\tilde\beta,\tilde\ga)$ with
\begin{align*}
10\tilde\al^2-\left(\frac{2}{k}\right)^2=1
\end{align*}
and $\tilde \beta=\frac{1}{3}$, $\tilde \ga=\frac{1}{3}$.

For the solutions given by $P=-w_1-2w_2-3w_3$, when $(\al,\beta,\ga)=(\frac{1}{k},\frac{1}{3},\frac{1}{2})$, we find $(\tilde\al,\tilde\beta,\tilde\ga)$ with
\begin{align*}
40\tilde\al^2-\left(\frac{2}{k}\right)^2=1,\hspace{12pt} \tilde\beta=\frac{1}{3},\hspace{12pt} \tilde\ga=\frac{1}{2}.
\end{align*}

When $(\al,\beta,\ga)=(\frac{1}{k},\frac{2}{k},\frac{1}{2})$, we find $(\tilde\al,\tilde\beta,\tilde\ga)$ with
\begin{align*}
40\tilde\al^2-\left(\frac{2}{k}\right)^2=1,\hspace{12pt} 10\tilde\beta^2-\left(\frac{2}{k}\right)^2=1,\hspace{12pt} \tilde\ga=\frac{1}{2}.
\end{align*}

When $(\al,\beta,\ga)=(\frac{1}{k},\frac{1}{3},\frac{3}{k})$, we find $(\tilde\al,\tilde\beta,\tilde\ga)$ with
\begin{align*}
40\tilde\al^2-\left(\frac{2}{k}\right)^2=1,\hspace{12pt} \tilde \beta=\frac{1}{3},\hspace{12pt} \frac{40}{9}\tilde\ga-\left(\frac{2}{k}\right)^2=1.
\end{align*}

Finally for the solutions given by $P=-4w_1-w_2-w_3$, when $(\al,\beta,\ga)=(\frac{4}{k},\frac{1}{k},\frac{1}{k})$, we find $(\tilde\al,\tilde\beta,\tilde\ga)$ with
\begin{align*}
\frac{5}{2}\tilde\al^2-\left(\frac{2}{k}\right)^2=1,\hspace{12pt} 40\tilde\beta^2-\left(\frac{2}{k}\right)^2=1,\hspace{12pt} 40\tilde\ga^2-\left(\frac{2}{k}\right)^2=1.
\end{align*}

When $(\al,\beta,\ga)=(\frac{2}{3},\frac{1}{k},\frac{1}{k})$, we find $(\tilde\al,\tilde\beta,\tilde\ga)$ with 
\begin{align*}
\tilde\al=\frac{2}{3},\hspace{12pt} 40\tilde\beta^2-\left(\frac{2}{k}\right)^2=1,\hspace{12pt} 40\tilde\ga^2-\left(\frac{2}{k}\right)^2=1.
\end{align*}
In all cases the appropriate substitution of $\tilde \al$, $\tilde \beta$ and $\tilde \ga$ in terms of the Chazy parameter $\tilde k$ gives equation (\ref{g2}), and therefore the equation for $u$ is the second-order differential equation associated to the generalised Chazy equation with parameter $\tilde k$, related to $k$ by (\ref{g2}). Altogether, with the exception of the parameters $k=\frac{3}{2}$ and $k=\frac{2}{3}$ as mentioned above, they give Ricci-flat but non-conformally flat examples of Nurowski's metric.

\qr

\end{document}